\documentclass[12pt,a4paper]{article}
\usepackage[utf8x]{inputenc}
\usepackage{amsmath}
\usepackage{amsfonts}
\usepackage{amssymb}
\usepackage{amsthm}
\usepackage{mathrsfs}
\usepackage{fancyhdr}
\usepackage{graphicx}
\usepackage[usenames,x11names]{xcolor}
\usepackage{arabtex}
\usepackage{framed}
\usepackage[top=2cm,bottom=2cm,margin=2cm]{geometry}
\usepackage{cancel}
\usepackage{array}   

\usepackage[pdftex,pdfstartview=FitH,colorlinks=true,linkcolor=Red4,urlcolor=Red4,%
filecolor=green,citecolor=red]{hyperref}

\title{On the least common multiple of binary linear recurrence sequences}
\author{\sc Sid Ali BOUSLA \\
Laboratoire de Mathématiques appliquées \\
Faculté des Sciences Exactes \\
Université de Bejaia, 06000 Bejaia, Algeria \\[1mm]
\href{mailto:bouslasidali@gmail.com}{bouslasidali@gmail.com} \\[1mm]
}
\date{}

\def\N{{\mathbb N}}

\def\lcm{\mathrm{lcm}}

\def\EMdash{\leavevmode\hbox to 10.6mm{\vrule height .63ex depth -.59ex
    width 10mm\hfill}}

\theoremstyle{plain}
\numberwithin{equation}{section}
\newtheorem{thm}{Theorem}[section]
\newtheorem{theorem}[thm]{Theorem}
\newtheorem{lemma}[thm]{Lemma}

\newtheorem{rmq}[thm]{Remark}

\newtheorem{coll}[thm]{Corollary}

\newcommand{\bi}{\genfrac{\{}{\}}{0pt}{}}

\begin{document}
\maketitle
\begin{abstract}
In this paper, we present a method for estimating the least common multiple of a large class of binary linear recurrence sequences. Let $P,Q,R_0$, and $R_1$ be fixed integers and let $\boldsymbol{R}=\left(R_n\right)_{n}$ be the recurrence sequence defined by $R_{n+2}=PR_{n+1}-QR_{n}$ $(\forall n\geq 0)$. Under some conditions on the parameters, we determine a rational nontrivial divisor for $L_{k,n}:=\lcm\left(R_k,R_{k+1},\dots,R_n\right)$, for all positive integers $n$ and $k$, such that $n\geq k$. As consequences, we derive nontrivial effective lower bounds for $L_{k,n}$ and we establish an asymptotic formula for $\log \left(L_{n,n+m}\right)$, where $m$ is a fixed positive integer. Denoting by $\left(F_n\right)_{n}$ the usual Fibonacci sequence, we prove for example that for any $m\geq 1$, we have
\[\log \lcm\left(F_{n},F_{n+1},\dots,F_{n+m}\right)\sim n(m+1)\log\Phi~~~~\text{as}~n\rightarrow +\infty,\]
where $\Phi$ denotes the golden ratio. We conclude the paper by some interesting identities and properties regarding the least common multiple of Lucas sequences.
\end{abstract}

\section{Introduction and Notation}
Throughout this paper, we let $\N^*$ denote the set $\N\setminus\lbrace 0\rbrace$ of positive integers. For $t\in\mathbb{R}$, we let $\lfloor t\rfloor$ and $\lceil t\rceil$ respectively denote the floor and the ceiling functions. For a given positive integer $n$ and given integers $a_1,a_2,\dots,a_n$ not all zero, we let $\gcd(a_1,a_2,\dots,a_n)$ and $\lcm(a_1,a_2,\dots,a_n)$ respectively denote the greatest common divisor and the least common multiple of $a_1,a_2,\dots,a_n$. We say that an integer $u$ is a multiple of a non-zero rational number $v$ (or equivalently, $v$ is a divisor of $u$) if the quotient $u/v$ is an integer. A sequence of non-zero integers $(a_n)_{n\geq 1}$ is said to be a \textit{divisibility sequence} if it satisfies, for all
$n,m\in\mathbb{N^*}$, the property: $n\mid m\Rightarrow a_n \mid a_m$. It is said to be a \textit{strong divisibility sequence} if it satisfies,
for all $n,m\in\mathbb{N^*}$, the stronger property: $\gcd(a_n,a_m)=\left|a_{\gcd(n,m)}\right|$. We shall use the Landau symbols $O$, $o$ and the Vinogradov symbols $\ll$, $\gg$ with their usual meanings.

On the studying of the distribution of primes, Chebyshev \cite{Chebyshev} showed that the prime number theorem is equivalent to stating that $\log\lcm(1,2,\dots,n){\sim}_{+\infty} n$. Since then, the estimates of the least common multiple of finite integer sequences became an important research problem. Recently, Hanson \cite{Hanson} and Nair \cite{Nair} respectively showed, by using elementary methods,  that $\lcm(1,2,\dots,n)\leq 3^n$ $(\forall n\in\N^*)$ and $\lcm(1,2,\dots,n)\geq 2^n$ $(\forall n\geq 7)$. In 2005, Farhi \cite{Farhi} proved that for all $u_0,r,n\in\mathbb{N^*}$ such that $\gcd(u_0,r)=1$, we have
\begin{equation}\label{arithprogr}
\lcm\left(u_0,u_0+r,\dots,u_0+nr\right)\geq u_0(1+r)^{n-1}.
\end{equation}   
More recently, Farhi \cite{Farhi4} managed to provide a $q$-analog of \eqref{arithprogr} by proving, for example, that for all $q,u_0,r\in\mathbb{N^*}$, such that $q\geq 2$ and $\gcd(u_0,r)=\gcd(u_0+r,q)=1$, we have
\begin{equation}\label{q-lcm}
\lcm\left(u_0+r{[1]}_{q},u_0+r{[2]}_{q},\dots,u_0+r{[n]}_{q}\right)\geq (u_0+r)A^{n-1}{q}^{\frac{(n-1)(n-4)}{4}},
\end{equation}
where ${[k]}_{q}:=\frac{q^k-1}{q-1}$ $(\forall k\in\mathbb{N^*})$, and $A>0$ is a computable constant depending on $q$, $u_0$, and $r$. 

In what follows, we let $\boldsymbol{R}=\left(R_n\right)_{n\in\mathbb{N}}$ be the binary recurrence sequence defined by:
\[R_{n}=PR_{n-1}-QR_{n-2}~~~~(\forall n\geq 2),\]
where $R_0,R_1,P$ and $Q$ are fixed integers such that $PQ\neq 0$, $\Delta:=P^2-4Q\neq 0$, and $\left|R_0\right|+\left|R_1\right|>0$. The number $\Delta$ is the discriminant of $\boldsymbol{R}$ and the companion polynomial of $\boldsymbol{R}$ is given by $P_{\boldsymbol{R}}(x)=x^2-Px+Q$. We denote by $\alpha$ and $\beta$ the distinct zeros of $P_{\boldsymbol{R}}$ such that $|\alpha|\geq |\beta|$. We have obviously $\alpha,\beta\neq 0$ and $|\alpha|\geq 1$ (since: $Q=\alpha\beta\in\mathbb{Z^*}$). It is well known that the terms $R_n$ can be written as:
\[R_n=a{\alpha}^n+b{\beta}^n~~~~(\forall n\in\mathbb{N}),\]
where $a:=\left(R_1-\beta R_0\right)/\left(\alpha-\beta\right)$ and $b:=\left(R_1-\alpha R_0\right)/\left(\beta-\alpha\right)$. Throughout the following, we assume that $\alpha/\beta$ is not a root of unity. The Lucas sequence (of the first kind) $U\left(P,Q\right)$ is given by:
\[U_n=\frac{{\alpha}^n-{\beta}^n}{\alpha-\beta}~~~~(\forall n\in\mathbb{N}),\]
and satisfies the recurrence relation: $U_{n+2}=PU_{n+1}-QU_n$ $(\forall n\in\mathbb{N})$, where $U_0=0$ and $U_1=1$. Note that $U\left(P,Q\right)$ is a divisibility sequence (see, e.g., \cite{Rib}, Eq. (2.10)). Moreover, if $P$ and $Q$ are coprime, then $U\left(P,Q\right)$ is a strong divisibility sequence (see, e.g., \cite{Rib}, Eq. (2.11)). For any natural number $j$, we let ${[j]}_{\boldsymbol{U}}!$ be the integer defined by: ${[j]}_{\boldsymbol{U}}!:=U_{1}U_{2}\cdots U_{j}$ (with the convention ${[0]}_{\boldsymbol{U}}!=1$). For all positive integers $n$ and $k$, such that $n\geq k$, we let ${\binom{n}{k}}_{\boldsymbol{U}}$ be the \textit{$\boldsymbol{U}$-binomial coefficient}, defined by:
\[{\binom{n}{k}}_{\boldsymbol{U}}:=\frac{U_nU_{n-1}\cdots U_{n-k+1}}{U_1U_2\cdots U_{k}}=\frac{{[n]}_{\boldsymbol{U}}!}{{[k]}_{\boldsymbol{U}}!{[n-k]}_{\boldsymbol{U}}!}.\]
Actually, if $P$ and $Q$ are coprime, these numbers are all integers (see, e.g., \cite{Bousla}, Prop. 1). By taking $(R_0,R_1)=(0,1)$, the sequence $\boldsymbol{R}$ is simply reduced to the Lucas sequence $U\left(P,Q\right)$; if in addition $(P,Q)=(1,-1)$ then $\boldsymbol{R}$ becomes the usual Fibonacci sequence defined by: $F_0=0$, $F_1=1$, and $F_{n+2}=F_{n+1}+F_n$ $(\forall n\in\mathbb{N})$.  

In 1986, Matiyasevich and Guy \cite{Mat} proved the interesting formula:
\begin{equation}\label{matiy}
\lim_{n\rightarrow +\infty}\frac{\log\left(F_1F_2\cdots F_n\right)}{\log\lcm\left(F_1,F_2,\dots,F_n\right)}=\frac{{\pi}^2}{6}.
\end{equation} 
In 1989, Kiss and M\'aty\'as \cite{kiss} showed that in \eqref{matiy} the Fibonacci sequence can be replaced by any sequence of the form $\left|U\left(P,Q\right)\right|$, with $\gcd\left(P,Q\right)=1$. In 1990, Akiyama \cite{ak1} proved more generally that for $R_0=0$ and $R_1\neq 0$, we have
\begin{equation}\label{akiy}
\lim_{n\rightarrow +\infty}\frac{\log\left|R_1R_2\cdots R_n\right|}{\log\lcm\left(R_1,R_2,\dots,R_n\right)}=\frac{{\pi}^2}{6(1-\kappa)},
\end{equation} 
where $\kappa:=\log\left(\gcd\left(P^2,Q\right)\right)/2\log |\alpha|$. In 2013, Akiyama and Luca \cite{ak4} managed to estimate the logarithm of the least common multiple of several Lucas subsequences. Quite recently, Sanna \cite{carlo2} showed that for any periodic sequence $\boldsymbol{s}=(s_n)_{n\geq 1}$ in $\lbrace−1,+1\rbrace$ there exists a positive computable rational number $C_{\boldsymbol{s}}$ such that:
\begin{equation}\label{saan}
\log\lcm(F_3+s_3,F_4+s_4,\dots,F_n+s_n){\sim}_{+\infty}\frac{3\log \Phi}{{\pi}^2}C_{\boldsymbol{s}}n^2~~~~\text{as}~n\rightarrow +\infty,
\end{equation}
where $\Phi$ denotes the golden ratio $(\Phi:=(1+\sqrt{5})/2)$. On the other hand, Bousla and Farhi \cite{Bousla,Bousfarcras} obtained several interesting identities dealing with the least common multiple of strong divisibility sequences. These latter deduced effective bounds for the least common multiple of Lucas sequences. In particular, they proved that for all $n\geq 1$, we have 
\begin{equation}
\Phi^{\frac{n^2}{4}-\frac{9}{4}}\leq\lcm\left(F_1,F_2,\dots,F_n\right)\leq\Phi^{\frac{n^2}{3}+\frac{4n}{3}}.
\end{equation}

For other types of sequences such as polynomial sequences or arithmetic progressions, there are also many works studying their least common multiple (see, e.g., \cite{bat,Bousfar1,arithpro,cil,Farhi,Hong}).

This paper is devoted to studying the numbers $L_{k,n}:=\lcm\left(R_k,R_{k+1},\dots,R_n\right)$, where $k$ and $n$ are positive integers such that $n\geq k$. Precisely, we determine a rational nontrivial divisor for $L_{k,n}$ and then we derive (as consequences) nontrivial lower bounds for $L_{k,n}$. We also establish an asymptotic formula for $\log \left(L_{n,n+m}\right)$, when $m$ is a fixed positive integer and $n$ tends to infinity. For Lucas sequences $U\left(P,Q\right)$, we obtain several identities dealing with their least common multiple. In fact, these identities are stronger than those given in \cite{Bousla,Bousfarcras} for the case of Lucas sequences. We close the paper by some concluding remarks and open questions.

\section{Main results}\label{sec2}

Let $P,Q,R_0$, and $R_1$ be fixed integers and suppose that $PQ,\Delta\neq 0$, $|R_0|+|R_1|>0$, and $\gcd\left(P,Q\right)=\gcd\left(R_1,Q\right)=1$. Our main results are given in the following:

\begin{theorem}\label{T1}
Let $n$ and $k$ be two positive integers such that $n\geq k$. Then the integer $L_{k,n}:=\lcm\left(R_k,R_{k+1},\dots,R_n\right)$ is a multiple of the rational number
\[\frac{R_k R_{k+1}\cdots R_n}{[n-k]_{\boldsymbol{U}}!\left(\gcd\left(R_0,R_1\right)\right)^{n-k}}.\]
\end{theorem}
\begin{rmq}
Theorem \ref{T1} generalizes the result of Farhi (\cite{Farhi4}, Theorem 1.1), obtained during the proof of \eqref{q-lcm} for the special case: $P=q+1$, $Q=q$, and $R_1>R_0\geq 1$, with $q\geq 2$ is an integer and $\gcd\left(R_0,R_1\right)=1$.
\end{rmq}

\begin{theorem}\label{T2}
Let $c,d\in\mathbb{N^*}$ be fixed and suppose that $P,Q\in\mathbb{N^*}$, and $\Delta>0$. Then for all $n,m\in\mathbb{N^*}$ such that $m\leq\left\lfloor \frac{n}{2}\right\rfloor$, we have
\[\lcm\left(cU_{m+1}+dU_{m},cU_{m+2}+dU_{m+1},\dots,cU_{n+1}+dU_{n}\right)\geq (c\alpha+d)\left(\frac{c\alpha+d}{\alpha\gcd(c,d)}\right)^{\frac{n}{2}} {\alpha}^{\frac{n^2}{4}}.\]
\end{theorem}

\begin{theorem}\label{T6}
Suppose that $P>0$, $Q<0$, and $R_0,R_1\in\mathbb{N^*}$. Then for all $n,m\in\mathbb{N^*}$ such that $n\geq 2$ and $m\leq\left\lfloor \frac{n+1}{2}\right\rfloor+1$, we have
\[\lcm\left(R_m,R_{m+1},\dots,R_n\right)\geq \gcd\left(R_0,R_1\right)\left(\frac{R_1+R_0\left|\beta\right|}{\gcd\left(R_0,R_1\right)}\right)^{\frac{n-1}{2}}{\alpha}^{\frac{n^2}{4}-\frac{n}{2}-\frac{7}{4}}.\] 
\end{theorem}

\begin{coll}\label{fibo}
Let $c$ and $d$ be two fixed positive integers. Then for all $n,m\in\mathbb{N^*}$ such that $n\geq 2$ and $m\leq \left\lfloor\frac{n+1}{2}\right\rfloor+1$, we have
\[\lcm\left(cF_{m}+dF_{m-1},cF_{m+1}+dF_{m},\dots,cF_{n}+dF_{n-1}\right)\geq \gcd(c,d)\left(\frac{c \Phi +d}{\Phi\gcd(c,d)}\right)^{\frac{n-1}{2}}{\Phi}^{\frac{n^2}{4}-\frac{n}{2}-\frac{7}{4}},\]
where $\Phi$ denotes the golden ratio $(\Phi:=(1+\sqrt{5})/2)$.
\end{coll}

\begin{theorem}\label{T3}
Let $q\geq 2$ be a fixed integer and $u_0,r\in\mathbb{N^*}$, with $\gcd\left(u_0+r,q\right)=1$. Let also $\left(u_n\right)_{n\geq 0}$ be the $q$-arithmetic sequence of parameters $u_0$ and $r$, that is $u_n=u_0+r{[n]}_{q}$ $(\forall n\in\mathbb{N^*})$. Then for all positive integers $n$ and $m$ such that $m\leq\left\lfloor \frac{n}{2}\right\rfloor$, we have
\[\lcm\left(u_m,u_{m+1},\dots,u_n\right)\geq \gcd(u_0,r)\left(\frac{r}{\gcd(u_0,r)}\right)^{\frac{n}{2}+1} {q}^{\frac{n(n-2)}{4}}.\]
\end{theorem}
\begin{rmq}
Theorem \ref{T3} gives an estimate of the same type as \eqref{q-lcm}, we prove it by a simpler way than the one given by Farhi \cite{Farhi4}. 
\end{rmq}

\begin{theorem}\label{T4}
There exists a computable constant $C>0$ such that:
\[\log \lcm\left(R_{\left\lfloor n/2\right\rfloor},R_{\left\lfloor n/2\right\rfloor+1},\dots,R_n\right)\geq \left(1 +o\left(\frac{1}{n}\right) \right)\frac{\log |\alpha|}{4} n^2~~~~(\forall n\geq C).\]
\end{theorem}

\begin{coll}\label{cor5}
Suppose that $R_n\neq 0$ $(\forall n\in\mathbb{N^*})$. We have
\[\log \lcm\left(R_1,R_2,\dots,R_n\right)\gg\log \left|R_1R_2\cdots R_n\right|.\]
\end{coll}

\begin{theorem}\label{T7}
Let $m$ be a fixed positive integer. Then:
\begin{equation}\label{T71}
\lim_{n\rightarrow +\infty}\frac{\log \lcm\left(R_n,R_{n+1},\dots,R_{n+m}\right)}{\log \left|R_n R_{n+1}\cdots R_{n+m}\right|}=1.
\end{equation}
In particular, we have
\begin{equation}\label{T72}
\log \lcm\left(R_n,R_{n+1},\dots,R_{n+m}\right)\sim n(m+1)\log |\alpha|~~~~\text{as}~n\rightarrow +\infty.
\end{equation}
\end{theorem}

\begin{coll}\label{fiii}
Let $m$ be a fixed positive integer. We have
\[\log \lcm\left(F_n,F_{n+1},\dots,F_{n+m}\right)\sim n(m+1)\log\Phi~~~~\text{as}~n\rightarrow +\infty.\]
\end{coll}

\begin{theorem}\label{co}
For all positive integers $n$ and $k$ such that $n\geq k$, we have
\begin{equation}\label{id}
\lcm\left(U_n,U_{n-1},\dots,U_{n-k+1}\right)=\lcm\left\lbrace U_m{\binom{n}{m}}_{\boldsymbol{U}};~1\leq m\leq k\right\rbrace.
\end{equation}
\end{theorem}

\begin{coll}\label{co2}
For all positive integers $n$ and $k$ such that $n\geq k$, we have
\begin{equation}\label{id2}
\lcm\left\lbrace {\binom{n}{m}}_{\boldsymbol{U}};~1\leq m\leq k\right\rbrace=\frac{\lcm\left(U_{n+1},U_n,\dots,U_{n-k+1}\right)}{\left|U_{n+1}\right|}.
\end{equation}
\end{coll}

\begin{rmq}
When $k=n$ in \eqref{id} and \eqref{id2}, we obtain the formulas of Bousla and Farhi \cite{Bousla,Bousfarcras} for the case of Lucas sequences. So, Theorem \ref{co} and Corollary \ref{co2} provide stronger identities for this case. It must be noted that the current identities are obtained by a simpler way, which do not involve the strong divisibility property. 
\end{rmq}

The following theorem generalize the results of Farhi \cite{lcmbinom}:

\begin{theorem}\label{tri}
Let $n$ and $k$ be two positive integers such that $n\geq k$. Then the number
\[{\bi{n}{k}}_{\boldsymbol{U}}:=\frac{\lcm\left(U_n,U_{n-1},\dots,U_{n-k+1}\right)}{\lcm\left(U_1,U_{2},\dots,U_{k}\right)},\]
is a positive integer. Moreover, we have
\[{\bi{n}{k}}_{\boldsymbol{U}}~\text{divides}~{\binom{n}{k}}_{\boldsymbol{U}}.\]
\end{theorem} 

\begin{coll}\label{derrr}
For any positive integer $n$, we have
\[\lcm\left(U_1,U_2,\dots,U_{n}\right)=\lcm\left(U_n,U_{n-1},\dots,U_{n-\left\lceil\frac{n}{2}\right\rceil+1}\right).\]
\end{coll}

\section{The proofs}

In what follows, we let $R_0,R_1,P$ and $Q$ be integers, with $|R_0|+|R_1|>0$, $PQ,\Delta\neq 0$, and $\gcd\left(P,Q\right)=\gcd\left(R_1,Q\right)=1$. For given positive integers $n$ and $k$ such that $n\geq k$, we define the following functions:
\begin{align*}
f(j,k,n)&:=\sum_{\begin{subarray}{c}k\leq i\leq n\\ i\neq j\end{subarray}}\min(i,j)~~~~(\forall j\in\left\lbrace k,k+1,\dots,n\right\rbrace),\\ g(k,n)&:=k+(k+1)+\dots +n=\frac{(n+k)(n-k+1)}{2},\\ h(k,n)&:=\max_{k\leq j\leq n}\left\lbrace f(j,k,n)\right\rbrace.
\end{align*} 
Since for any $j\in\left\lbrace k,k+1,\dots,n\right\rbrace$ the number $\frac{{[n-k]}_{\boldsymbol{U}}!}{{[j-k]}_{\boldsymbol{U}}!{[n-j]}_{\boldsymbol{U}}!}={\binom{n-k}{n-j}}_{\boldsymbol{U}}$ is an integer, then the following relations hold:
\begin{equation}\label{e15}
\lcm\left\lbrace Q^{f(j,k,n)}{[j-k]}_{\boldsymbol{U}}!{[n-j]}_{\boldsymbol{U}}!;~j=k,k+1,\dots,n\right\rbrace~~\text{divides}~~Q^{h(k,n)}{[n-k]}_{\boldsymbol{U}}!,
\end{equation}
and
\begin{equation}\label{eeeee}
\lcm\left\lbrace Q^{f(j,k,n)-(n-k)}{[j-k]}_{\boldsymbol{U}}!{[n-j]}_{\boldsymbol{U}}!;~j=k,k+1,\dots,n\right\rbrace~~\text{divides}~~Q^{h(k,n)}{[n-k]}_{\boldsymbol{U}}!.
\end{equation}
Now, for all positive integers $k$ and $k'$ such that $k\leq k'$, we have obviously: $L_{k',n}$ divides $L_{k,n}$; so, we shall use the following inequality:
\begin{equation}\label{inc}
L_{k',n}\leq L_{k,n}~~~~(\forall k,k'\in\mathbb{N^*},~k\leq k').
\end{equation}

\subsection{Proof of Theorem \ref{T1}}

The proof of Theorem \ref{T1} needs the following lemmas: 

\begin{lemma}\label{L2}
For a given positive integer $m$ and given non-zero pairwise distinct complex numbers $z_1,z_2,\dots,z_m$, we have
\[\frac{1}{z_1 z_2\cdots z_m}=\sum_{j=1}^{m}\frac{1}{\prod_{\begin{subarray}{c} 1\leq i\leq m\\i\neq j\end{subarray}}(z_i-z_j)}\cdot\frac{1}{z_j}.\]
\end{lemma}
\begin{proof}
The lemma follows by expressing $1/(z+z_1)(z+z_2)\cdots (z+z_m)$ in partial fractions and taking $z=0$.
\end{proof}

\begin{lemma}\label{eqdeR}
For any positive integer $m$, we have
\[R_m=R_1U_m-R_0QU_{m-1}.\]
\end{lemma}
\begin{proof}
Let $m$ be a positive integer. We have
\begin{align*}
R_m&=a{\alpha}^m+b{\beta}^m\\&=\left(\frac{R_1-\beta R_0}{\alpha-\beta}\right){\alpha}^m - \left(\frac{R_1-\alpha R_0}{\alpha-\beta}\right){\beta}^m\\&=R_1\left(\frac{{\alpha}^m-{\beta}^m}{\alpha-\beta}\right)-R_0\left(\frac{\beta{\alpha}^m-\alpha{\beta}^m}{\alpha-\beta}\right)\\&=R_1\left(\frac{{\alpha}^m-{\beta}^m}{\alpha-\beta}\right)-R_0(\alpha\beta)\left(\frac{{\alpha}^{m-1}-{\beta}^{m-1}}{\alpha-\beta}\right)\\&=R_1U_m-R_0QU_{m-1},
\end{align*}
as required. This completes the proof of the lemma.
\end{proof}

\begin{lemma}\label{dif}
For all positive integers $i$ and $j$, we have
\[\frac{U_{j-1}}{U_{j}}-\frac{U_{i-1}}{U_{i}}=\frac{U_{i}U_{j-1}-U_{i-1}U_{j}}{U_{i}U_{j}}=\begin{cases} -\frac{Q^{j-1}U_{i-j}}{U_i U_{j}}~~&\text{if}~i\geq j\\ \frac{Q^{i-1}U_{j-i}}{U_i U_{j}}~~&\text{otherwise}\end{cases}.\] 
Furthermore, if $i,j\geq 2$, then:
\[\frac{U_{i}}{U_{i-1}}-\frac{U_{j}}{U_{j-1}}=\frac{U_{i}U_{j-1}-U_{i-1}U_{j}}{U_{i-1}U_{j-1}}=\begin{cases} -\frac{Q^{j-1}U_{i-j}}{U_{i-1} U_{j-1}}~~&\text{if}~i\geq j\\ \frac{Q^{i-1}U_{j-i}}{U_{i-1} U_{j-1}}~~&\text{otherwise}\end{cases}.\]
\end{lemma}
\begin{proof}
Let $i$ and $j$ be two fixed positive integers. First, suppose that $i\geq j$. We have
\begin{align*}
U_{i}U_{j-1}-U_{i-1}U_{j}&=\frac{\left({\alpha}^{i}-{\beta}^{i}\right)\left({\alpha}^{j-1}-{\beta}^{j-1}\right)-\left({\alpha}^{i-1}-{\beta}^{i-1}\right)\left({\alpha}^{j}-{\beta}^{j}\right)}{(\alpha-\beta)^2}\\&=-\frac{{\alpha}^{i}{\beta}^{j-1}+{\alpha}^{j-1}{\beta}^{i}-{\alpha}^{i-1}{\beta}^{j}-{\alpha}^{j}{\beta}^{i-1}}{(\alpha-\beta)^2}\\&=-(\alpha\beta)^{j-1}\frac{\left(\alpha-\beta\right)\left({\alpha}^{i-j}-{\beta}^{i-j}\right)}{(\alpha-\beta)^2}\\&=-Q^{j-1}U_{i-j},
\end{align*}
which concludes to the required result. The case when $i<j$ follows by permuting $i$ and $j$ in the first case. This completes the proof. 
\end{proof}

\begin{lemma}\label{prod}
Let $n$ and $k$ be two positive integers such that $n\geq k$. Then for any $j\in\lbrace k,k+1,\dots,n\rbrace$, we have
\[\prod_{\begin{subarray}{c}k\leq i\leq n\\i\neq j\end{subarray}}\left( R_0 Q\left(\frac{U_{j-1}}{U_{j}}-\frac{U_{i-1}}{U_{i}}\right)\right)=(-1)^{n-j}{R}_{0}^{n-k} Q^{f(j,k,n)}\frac{{[j-k]}_{\boldsymbol{U}}!{[n-j]}_{\boldsymbol{U}}!}{{U}_{j}^{n-k-1}\left(U_kU_{k+1}\cdots U_{n}\right)}.\]
Furthermore, if $k\geq 2$, we have
\[\prod_{\begin{subarray}{c} k\leq i\leq n\\i\neq j\end{subarray}}\left( R_1\left(\frac{U_{i}}{U_{i-1}}-\frac{U_{j}}{U_{j-1}}\right)\right)=(-1)^{n-j} {R}_{1}^{n-k}Q^{f(j,k,n)-(n-k)}\frac{{[j-k]}_{\boldsymbol{U}}!{[n-j]}_{\boldsymbol{U}}!}{{U}_{j-1}^{n-k-1}\left(U_{k-1}U_{k}\cdots U_{n-1}\right)}.\]
\end{lemma}
\begin{proof}
This is an immediate consequence of Lemma \ref{dif}.
\end{proof}

\begin{lemma}\label{4}
Let $n$ and $k$ be two positive integers such that $n\geq k$ and $n\geq 2$. Suppose that all the numbers $R_k,R_{k+1},\dots,R_n$ are not zero. If $R_0\neq 0$, then:
\begin{equation}\label{e5}
\frac{R_{0}^{n-k}}{R_kR_{k+1}\cdots R_n}=\sum_{j=k}^{n}\frac{(-1)^{n-j}{U}_{j}^{n-k}}{Q^{f(j,k,n)}{[j-k]}_{\boldsymbol{U}}!{[n-j]}_{\boldsymbol{U}}!}\cdot\frac{1}{R_j}.
\end{equation}
Moreover, if $R_1\neq 0$, then:
\begin{equation}\label{e6}
\frac{{R}_{1}^{n-k}}{R_kR_{k+1}\cdots R_n}=\sum_{j=k}^{n}\frac{(-1)^{n-j}{U}_{j-1}^{n-k}}{ Q^{f(j,k,n)-(n-k)}{[j-k]}_{\boldsymbol{U}}!{[n-j]}_{\boldsymbol{U}}!}\cdot\frac{1}{R_j}. 
\end{equation} 
\end{lemma}
\begin{proof}
Assume that $R_0\neq 0$ and let us show that \eqref{e5} holds. By applying Lemma \ref{L2} to the finite sequence $z_i:=\frac{R_{i+k-1}}{U_{i+k-1}}=R_1-R_0Q\frac{U_{i+k-2}}{U_{i+k-1}}$ ($i=1,2,\dots,n-k+1$) and using Lemmas \ref{eqdeR} and \ref{prod}, we obtain
\begin{align*}
\frac{U_kU_{k+1}\cdots U_n}{R_kR_{k+1}\cdots R_n}&=\sum_{j=k}^{n}\frac{1}{\prod_{\begin{subarray}{c} k\leq i\leq n\\i\neq j\end{subarray}}\left(R_0Q\left(\frac{U_{j-1}}{U_{j}}-\frac{U_{i-1}}{U_{i}}\right)\right)}\cdot\frac{U_j}{R_j}\\&=\sum_{j=k}^{n}\frac{(-1)^{n-j}{U}_{j}^{n-k}\left(U_kU_{k+1}\cdots U_{n}\right)}{R_{0}^{n-k}Q^{f(j,k,n)}{[j-k]}_{\boldsymbol{U}}!{[n-j]}_{\boldsymbol{U}}!}\cdot\frac{1}{R_j}.
\end{align*}
Equivalently, we have
\[\frac{R_{0}^{n-k}}{R_kR_{k+1}\cdots R_n}=\sum_{j=k}^{n}\frac{(-1)^{n-j}{U}_{j}^{n-k}}{Q^{f(j,k,n)}{[j-k]}_{\boldsymbol{U}}!{[n-j]}_{\boldsymbol{U}}!}\cdot\frac{1}{R_j},\]
which confirms the first part of the lemma. Next, suppose that $R_1\neq 0$ and let us show that \eqref{e6} holds. First, assume that $k\geq 2$ (to avoid the case when $U_{k-1}=0$). By applying Lemma \ref{L2} to the finite sequence $z'_{i}:=\frac{R_{i+k-1}}{U_{i+k-2}}=R_1\frac{U_{i+k-1}}{U_{i+k-2}}-R_0Q$ ($i=1,2,\dots,n-k+1$), we get (according to the lemmas \ref{eqdeR} and \ref{prod})
\begin{align*}
\frac{U_{k-1}U_{k}\cdots U_{n-1}}{R_kR_{k+1}\cdots R_n}&=\sum_{j=k}^{n}\frac{1}{\prod_{\begin{subarray}{c} k\leq i\leq n\\i\neq j\end{subarray}}\left(R_1\left(\frac{U_{i}}{U_{i-1}}-\frac{U_{j}}{U_{j-1}}\right)\right)}\cdot\frac{U_{j-1}}{R_{j}}\\&=\sum_{j=k}^{n}\frac{(-1)^{n-j}{U}_{j-1}^{n-k}\left(U_{k-1}U_{k}\cdots U_{n-1}\right)}{{R}_{1}^{n-k} Q^{f(j,k,n)-(n-k)}{[j-k]}_{\boldsymbol{U}}!{[n-j]}_{\boldsymbol{U}}!}\cdot\frac{1}{R_j},
\end{align*}
which is equivalent to saying that:
\begin{equation}\label{suppp}
\frac{{R}_{1}^{n-k}}{R_kR_{k+1}\cdots R_n}=\sum_{j=k}^{n}\frac{(-1)^{n-j}{U}_{j-1}^{n-k}}{ Q^{f(j,k,n)-(n-k)}{[j-k]}_{\boldsymbol{U}}!{[n-j]}_{\boldsymbol{U}}!}\cdot\frac{1}{R_j},
\end{equation} 
as required. The case when $k=1$ follows by observing that for any $n\geq 2$, we have
\begin{align*}
\frac{{R}_{1}^{n-1}}{R_1R_{2}\cdots R_n}&=\frac{{R}_{1}^{n-2}}{R_{2}R_{3}\cdots R_n}\\&=\sum_{j=2}^{n}\frac{(-1)^{n-j}{U}_{j-1}^{n-2}}{ Q^{f(j,2,n)-(n-2)}{[j-2]}_{\boldsymbol{U}}!{[n-j]}_{\boldsymbol{U}}!}\cdot\frac{1}{R_j}~~~~(\text{according to \eqref{suppp}})\\&=\sum_{j=1}^{n}\frac{(-1)^{n-j}{U}_{j-1}^{n-1}}{ Q^{f(j,1,n)-(n-1)}{[j-1]}_{\boldsymbol{U}}!{[n-j]}_{\boldsymbol{U}}!}\cdot\frac{1}{R_j},
\end{align*}
(since $U_0=0$, $f(j,1,n)-(n-1)=f(j,2,n)-(n-2)$, and ${[j-2]}_{\boldsymbol{U}}!U_{j-1}={[j-1]}_{\boldsymbol{U}}!$ $(\forall j\geq 2)$). This confirms the required result and completes the proof of the lemma.
\end{proof}

\begin{lemma}\label{lucas}
For any positive integer $m$, we have $\gcd\left(U_m,Q\right)=1$.
\end{lemma}
\begin{proof}
See, e.g., \cite[Eq. (2.14)]{Rib}.
\end{proof}

\begin{lemma}\label{gcd1}
For any positive integer $m$, we have $\gcd\left(R_m,Q\right)=1$.
\end{lemma}
\begin{proof}
Let $m$ be a fixed positive integer. We will show that the only positive common divisor of $R_m$ and $Q$ is $1$, which concludes to the required result. Suppose that $d$ is an arbitrary positive common divisor of $R_m$ and $Q$. The number $d$ is then (according to Lemma \ref{eqdeR}) a divisor of the number $R_m+R_0QU_{m-1}=R_1U_m$. On the other hand, since $\gcd\left(U_m,Q\right)=1$ (by Lemma \ref{lucas}), it follows that $d$ and $U_m$ are relatively prime; so, $d$ divides $R_1$ (from the Gauss lemma) and $Q$. Next, since $\gcd(R_1,Q)=1$ (by hypothesis), then $d=1$, as required. The lemma is proved.
\end{proof}

We are now ready to prove Theorem \ref{T1}:

\begin{proof}[Proof of Theorem \ref{T1}]
If $n=1$ or one of the numbers $R_k,R_{k+1},\dots,R_n$ is zero, the result of the theorem is trivial. Suppose for the sequel that $n\geq 2$ and all the numbers $R_k,R_{k+1},\dots,R_n$ are not zero. We distinguish the following two cases:\\
\underline{1\textsuperscript{st}case:} (if $R_0,R_1\neq 0$).\\
By multiplying the two sides of \eqref{e5} and \eqref{e6} by $Q^{h(k,n)}{[n-k]}_{\boldsymbol{U}}!L_{k,n}$, we obtain (according to \eqref{e15} and \eqref{eeeee}):
\[\frac{Q^{h(k,n)}{[n-k]}_{\boldsymbol{U}}!L_{k,n}R_{0}^{n-k}}{R_k R_{k+1}\cdots R_n},\frac{Q^{h(k,n)}{[n-k]}_{\boldsymbol{U}}!L_{k,n}R_{1}^{n-k}}{R_k R_{k+1}\cdots R_n}\in\mathbb{Z}.\]
Consequently, we have
\[R_k R_{k+1}\cdots R_n~~\text{divides}~~Q^{h(k,n)}{[n-k]}_{\boldsymbol{U}}!L_{k,n}\left(\gcd\left(R_0,R_1\right)\right)^{n-k}.\]
Combining this with Lemma \ref{gcd1}, we get (according to the Gauss lemma):
\[R_k R_{k+1}\cdots R_n~~\text{divides}~~{[n-k]}_{\boldsymbol{U}}!L_{k,n}\left(\gcd\left(R_0,R_1\right)\right)^{n-k},\]
as required.\\
\underline{2\textsuperscript{nd}case:} (if $R_0=0$ or $R_1=0$).\\
In this case, one of the numbers $R_0$ or $R_1$ is not zero; so, one of the identities \eqref{e5} or \eqref{e6} holds. This immediately implies as above that:
\[R_k R_{k+1}\cdots R_n~~\text{divides}~~Q^{h(k,n)}{[n-k]}_{\boldsymbol{U}}!L_{k,n}R_{i}^{n-k},\]
for some $i\in\lbrace 0,1\rbrace$ such that $R_i\neq 0$. It then follows from Lemma \ref{gcd1} that:
\[R_kR_{k+1}\cdots R_n~~\text{divides}~~{[n-k]}_{\boldsymbol{U}}!L_{k,n} R_{i}^{n-k}={[n-k]}_{\boldsymbol{U}}!L_{k,n}\left(\gcd\left({R}_{0},R_1\right)\right)^{n-k},\]
which concludes to the required result and achieves the proof of the theorem.
\end{proof}

\subsection{Proof of Theorem \ref{T2}}

We shall need the following three elementary lemmas:

\begin{lemma}\label{terrr}
Suppose that $\Delta>0$. Then for any positive integer $t$, we have
\begin{equation}\label{ut}
\left|\alpha\right|^{t-2}\leq\left|U_t\right|\leq\left|\alpha\right|^{t}.
\end{equation}
In addition, the right-hand side inequality of \eqref{ut} holds for any $\Delta\neq 0$.
\end{lemma}
\begin{proof}
See, e.g., \cite[Lemma 2.7]{Bousla}.
\end{proof}

\begin{lemma}\label{estU2}
Suppose that $\Delta>0$. Then for all positive integers $m$ and $\ell$ such that $m\geq \ell$, we have
\[ \left|U_{\ell}\binom{m}{\ell}_{\bf{U}}\right|\geq\left|\alpha\right|^{\ell\left(m-\ell-1\right)-1}.\]
\end{lemma}
\begin{proof}
This is an immediate consequence of Lemma \ref{terrr}.
\end{proof}

\begin{lemma}\label{p1p2}
Suppose that $P,Q\in\mathbb{N^*}$. Then:
\begin{enumerate}
\item We have $\alpha>\beta>0$ and $U_m\geq 1$ $(\forall m\geq 1)$.\label{pp1}
\item For any natural number $m$, we have $U_{m+1}\geq \alpha U_m$.\label{pp2}
\end{enumerate}
\end{lemma}
\begin{proof}
The first point of the lemma is immediate. Let $m\in\mathbb{N}$ be fixed and let us prove the point \ref{pp2} of the lemma. For $m=0$, the result is trivial. Suppose for the sequel that $m\geq 1$. According to the point \ref{pp1} of the lemma, we have
\begin{align*}
\frac{U_{m+1}}{U_{m}}=\frac{{\alpha}^{m+1}-{\beta}^{m+1}}{{\alpha}^{m}-{\beta}^{m}}=\frac{\alpha\left({\alpha}^{m}-{\beta}^{m}\right)+\alpha {\beta}^m-{\beta}^{m+1}}{{\alpha}^{m}-{\beta}^{m}}=\alpha + \frac{{\beta}^m(\alpha-\beta)}{{\alpha}^{m}-{\beta}^{m}}\geq \alpha,
\end{align*} 
as required. This completes the proof.
\end{proof}

Now, we are ready to prove Theorem \ref{T2}:

\begin{proof}[Proof of Theorem \ref{T2}]
By taking $R_0=c$ and $R_1=cP+d$, we have clearly $\gcd\left(R_0,R_1\right)=\gcd(c,d)$. Next, according to Lemma \ref{eqdeR}, we have for any positive integer $t$:
\[R_t=R_1U_t-R_0QU_{t-1}=\left(cP+d\right)U_t-cQU_{t-1}=c\left(PU_t-QU_{t-1}\right)+dU_{t}=cU_{t+1}+dU_{t}.\]
Combining this with Lemma \ref{p1p2}, we get $R_t\geq (c\alpha+d)U_t$ $(\forall t\in\mathbb{N^*})$. It then follows from Theorem \ref{T1}, Inequality \eqref{inc}, and Lemma \ref{estU2} that for all $n,m\in\mathbb{N^*}$ such that $m\leq\left\lfloor \frac{n}{2}\right\rfloor$, we have
\begin{align*}
\lcm\left(R_m,R_{m+1},\dots,R_n\right)&\geq \max_{m\leq k\leq n}\left\lbrace\frac{R_kR_{k+1}\cdots R_n}{U_1U_2\cdots U_{n-k}\left(\gcd\left(R_0,R_1\right)\right)^{n-k}}\right\rbrace\\&\geq \gcd(c,d)\max_{m\leq k\leq n}\left\lbrace{\left(\frac{c\alpha+d}{\gcd(c,d)}\right)}^{n-k+1}\frac{U_{k}U_{k+1}\cdots U_{n}}{U_1U_2\cdots U_{n-k}}\right\rbrace\\&= \gcd(c,d)\max_{m\leq k\leq n}\left\lbrace {\left(\frac{c\alpha+d}{\gcd(c,d)}\right)}^{n-k+1} U_{k}{\binom{n}{k}}_{\boldsymbol{U}}\right\rbrace
\\&\geq \gcd(c,d)\max_{m\leq k\leq n}\left\lbrace{\left(\frac{c\alpha+d}{\gcd(c,d)}\right)}^{n-k+1}{\alpha}^{k(n-k-1)-1}\right\rbrace
\\&\geq  \gcd(c,d) {\left(\frac{c\alpha+d}{\gcd(c,d)}\right)}^{n-\lfloor n/2\rfloor+1} {\alpha}^{\lfloor n/2\rfloor(n-\lfloor n/2\rfloor-1)-1}
\\&\geq \gcd(c,d) {\left(\frac{c\alpha+d}{\gcd(c,d)}\right)}^{n/2+1} {\alpha}^{n^2/4-n/2+\left(n/2-\lfloor n/2\rfloor \right)-\left(n/2-\lfloor n/2\rfloor \right)^2}
\\&\geq \gcd(c,d){\left(\frac{c\alpha+d}{\gcd(c,d)}\right)}^{n/2+1} {\alpha}^{n^2/4-n/2},
\end{align*}
which concludes to the required result and completes the proof of the theorem.
\end{proof}


\subsection{Proofs of Theorems \ref{T6} and \ref{T3} and Corollary \ref{fibo}}

We shall use the following lemma:

\begin{lemma}\label{min}
Suppose that $P>0$, $Q<0$, and $R_0,R_1\in\mathbb{N^*}$. Then we have
\[R_m\geq \left(R_1+R_0\left|\beta\right|\right)\alpha^{m-2}~~~~(\forall m\geq 2).\]
\end{lemma}
\begin{proof}
First, it is easy to show that $\beta<0$, $\alpha>|\beta|>0$, and $U_t\geq 1$ $(\forall t\in\mathbb{N^*})$; so, by Lemma \ref{terrr}, we have $U_t\geq {\alpha}^{t-2}$ $(\forall t\in\mathbb{N^*})$. This with Lemma \ref{eqdeR} imply that for any integer $m\geq 2$, we have
\begin{align*}
R_m&=R_1U_m-R_0QU_{m-1}\\&=R_1U_m-R_0\alpha\beta U_{m-1}\\&= R_1U_m+R_0\alpha\left|\beta\right|U_{m-1}\\&\geq \left(R_1+R_0\left|\beta\right|\right)\alpha^{m-2},
\end{align*}
which concludes to the required result. This completes the proof.
\end{proof}

\begin{proof}[Proof of Theorem \ref{T6}]
Let $n\geq 2$ be an integer and put ${\ell}_n:=\left\lfloor\frac{n+1}{2}\right\rfloor+1$. By applying successively Theorem \ref{T1}, Inequality \eqref{inc}, and Lemma \ref{min}, we get for any $m\in\lbrace 1,2,\dots,{\ell}_n\rbrace$:
\begin{align*}
\lcm\left(R_m,R_{m+1},\dots,R_n\right)&\geq \max_{m\leq k\leq n}\left\lbrace\frac{R_kR_{k+1}\cdots R_n}{U_1U_2\dots U_{n-k}{\left(\gcd\left(R_0,R_1\right)\right)}^{n-k}}\right\rbrace
\\&\geq \max_{{\ell}_n\leq k\leq n}\left\lbrace\frac{R_kR_{k+1}\cdots R_n}{U_1U_2\dots U_{n-k}{\left(\gcd\left(R_0,R_1\right)\right)}^{n-k}}\right\rbrace
\\&\geq \gcd\left(R_0,R_1\right)\max_{{\ell}_n\leq k\leq n}\left\lbrace\left(\frac{R_1+R_0\left|\beta\right|}{\gcd\left(R_0,R_1\right)}\right)^{n-k+1}\frac{\alpha^{(k-2)+(k-1)+\dots+(n-2)}}{\alpha^{1+2+\dots+(n-k-1)}}\right\rbrace
\\&=\gcd\left(R_0,R_1\right)\max_{{\ell}_n\leq k\leq n}\left\lbrace\left(\frac{R_1+R_0\left|\beta\right|}{\gcd\left(R_0,R_1\right)}\right)^{n-k+1}{\alpha}^{k(n-k+2)-(n+2)}\right\rbrace
\\&\geq \gcd\left(R_0,R_1\right)\left(\frac{R_1+R_0\left|\beta\right|}{\gcd\left(R_0,R_1\right)}\right)^{n-{\ell}_n+1}{\alpha}^{{\ell}_n(n-{\ell}_n+2)-(n+2)}
\\&\geq \gcd\left(R_0,R_1\right)\left(\frac{R_1+R_0\left|\beta\right|}{\gcd\left(R_0,R_1\right)}\right)^{\frac{n-1}{2}}{\alpha}^{\frac{n^2}{4}-\frac{n}{2}-\frac{7}{4}},
\end{align*}
which confirms the required result and achieves the proof of the theorem.
\end{proof}

\begin{proof}[Proof of Corollary \ref{fibo}]
By taking $P=1$, $Q=-1$, $R_0=d$, and $R_1=c$, we have (according to Lemma \ref{eqdeR}): $R_n=cF_n+dF_{n-1}$ $(\forall n\in\mathbb{N^*})$. On the other hand, all conditions of Theorem \ref{T6} are satisfied. So, Theorem \ref{T6} concludes then to the required result.
\end{proof}

\begin{proof}[Proof of Theorem \ref{T3}]
By putting $P:=q+1$, $Q:=q$, $R_0=u_0$, and $R_1=u_0+r$, we get $\alpha=q$, $\beta=1$, $U_t={[t]}_{q}:=\frac{{q}^t-1}{q-1}$, and (according to Lemma \ref{eqdeR}) $R_t=u_t$ $(\forall t\in\mathbb{N})$. On the other hand, for any positive integer $t$, we have $R_{t}\geq r{[t]}_{q}$. Combining this with Theorem \ref{T1} and Lemma \ref{estU2}, we obtain (according to \eqref{inc}), for all positive integers $n$ and $m$, with $m\leq \left\lfloor \frac{n}{2}\right\rfloor$:
\begin{align*}
\lcm\left(u_m,u_{m+1},\dots,u_n\right)&\geq \max_{m\leq k\leq n}\left\lbrace\frac{u_ku_{k+1}\cdots u_n}{[1]_q[2]_q\cdots [n-k]_{q}\left(\gcd\left(u_0,u_0+r\right)\right)^{n-k}}\right\rbrace 
\\&\geq \gcd(u_0,r)\max_{m\leq k\leq n}\left\lbrace\left(\frac{r}{\gcd(u_0,r)}\right)^{n-k+1}\frac{{[k]}_{q}{[k+1]}_{q}\cdots {[n]}_{q}}{{[1]}_{q}{[2]}_{q}\cdots {[n-k]}_{q}}\right\rbrace 
\\&=\gcd(u_0,r)\max_{m\leq k\leq n} \left\lbrace \left(\frac{r}{\gcd(u_0,r)}\right)^{n-k+1}{[k]}_{q}{\binom{n}{k}}_{q}\right\rbrace 
\\&\geq \gcd(u_0,r)\max_{m\leq k\leq n}\left\lbrace \left(\frac{r}{\gcd(u_0,r)}\right)^{n-k+1} {q}^{k(n-k-1)-1}\right\rbrace 
\\&\geq \gcd(u_0,r)\left(\frac{r}{\gcd(u_0,r)}\right)^{n-\lfloor n/2\rfloor+1} {q}^{\lfloor n/2\rfloor(n-\lfloor n/2\rfloor-1)-1}
\\&\geq \gcd(u_0,r)\left(\frac{r}{\gcd(u_0,r)}\right)^{n/2+1} {q}^{n^2/4-n/2+\left(n/2-\lfloor n/2\rfloor \right)-\left(n/2-\lfloor n/2\rfloor \right)^2}
\\&\geq \gcd(u_0,r)\left(\frac{r}{\gcd(u_0,r)}\right)^{n/2+1} {q}^{n^2/4-n/2}.
\end{align*}
This confirms the required result and achieves the proof of the theorem.
\end{proof}

\subsection{Proofs of Theorems \ref{T4} and \ref{T7} and their corollaries}

In order to simplify some statements, we set $T:=\log\gcd\left(R_0,R_1\right)$. We shall need the following well-known lemma:

\begin{lemma}[Shorey-Stewart \cite{Shorey}]\label{sho}
There exist two positive constants $c_0$ and $c_1$, which are effectively computable in terms of $a$ and $b$, such that for any integer $t\geq c_1$, we have
\[\left|R_t\right|\geq |\alpha|^{t-c_0\log t}.\]
\end{lemma}
\begin{proof}
This is Lemma 5 of \cite{Shorey}, obtained as a consequence of the Baker method \cite{Baker}.
\end{proof}

\begin{proof}[Proof of Theorem \ref{T4}]
Let $n\geq C:=2(c_1+1)$ be a fixed integer and put $k_n:=\left\lfloor \frac{n}{2}\right\rfloor$. According to Lemma \ref{sho} and the easy inequality $k_n\binom{n}{k_n}\leq n2^n$, we have
\begin{align*}
\log \left|R_{k_n}R_{k_n+1}\cdots R_n\right|&\geq \left((k_n-c_0\log k_n)+\dots +(n-c_0\log n)\right) \log |\alpha|\\&\geq \left(\frac{(n+k_n)(n-k_n+1)}{2} - c_0\log\left(k_n\binom{n}{k_n}\right)\right) \log |\alpha|\\&\geq  \left(\frac{(n+k_n)(n-k_n+1)}{2} - c_0\log n -c_0n\log 2\right) \log |\alpha|.
\end{align*}
On the other hand, since $\left|U_t\right|\leq |\alpha|^t$ $(\forall t\in\mathbb{N^*})$ (by Lemma \ref{terrr}), we have
\begin{align*}
\log \left|U_1U_2\dots U_{n-k_n}\right|&\leq \left(1+2+\dots+(n-k_n)\right)\log |\alpha|\\&\leq \frac{(n-k_n)(n-k_n+1)}{2}\log |\alpha|. 
\end{align*}
Combining this with Theorem \ref{T1}, we get 
\begin{align*}
\log \lcm\left(R_{k_n},R_{k_n+1},\dots,R_n\right)&\geq \log \left|R_{k_n}R_{k_n+1}\cdots R_n\right|-\log \left|U_1U_2\dots U_{n-k_n}\right|-(n-k_n)T\\&\geq \left(k_n(n-k_n+1)- c_0\log n -c_0n\log 2\right) \log |\alpha|-(n-k_n)T\\&\geq \left(\frac{n(n-2)}{4}- c_0\log n -c_0n\log 2\right) \log |\alpha|-\left(n-\left\lfloor \frac{n}{2}\right\rfloor\right)T\\&\geq \left(1+o\left(\frac{1}{n}\right)\right)\frac{\log |\alpha|}{4} n^2.
\end{align*}
This confirms the required result and completes the proof of the theorem.
\end{proof}

\begin{proof}[Proof of Corollary \ref{cor5}]
Since $|\alpha|\geq|\beta|$, then for any positive integer $t$, we have
\begin{equation}\label{eee}
\left|R_t\right|=\left|a{\alpha}^{t}+b{\beta}^{t}\right|\leq |a|{|\alpha|}^t+|b|{|\beta|}^t\leq (|a|+|b|){|\alpha|}^t.
\end{equation} 
Consequently, we have for any $n\in\mathbb{N^*}$:
\[\left|R_1R_2\cdots R_n\right|\leq \left(|a|+|b|\right)^n\left|\alpha\right|^{1+2+\dots+n}=\left(|a|+|b|\right)^n\left|\alpha\right|^{\frac{n(n+1)}{2}}.\]
Hence:
\[\log \left|R_1R_2\cdots R_n\right|\leq n\log \left(|a|+|b|\right)+\frac{n(n+1)}{2}\log\left|\alpha\right|\ll \left(1+o\left(\frac{1}{n}\right)\right)\frac{\log |\alpha|}{4}n^2.\]
The required estimate then follows from Theorem \ref{T4} and \eqref{inc}. This completes the proof.
\end{proof}

\begin{proof}[Proof of Theorem \ref{T7}]
Let $m$ be a fixed positive integer and let us first prove \eqref{T71}. Let also $n$ be a sufficiently large integer such that $R_t\neq 0$ $(\forall t\geq n)$ (the existence of such $n$ is assured by Lemma \ref{sho}). On the one hand, we have
\[\frac{\log \lcm\left(R_n,R_{n+1},\dots,R_{n+m}\right)}{\log \left|R_n R_{n+1}\cdots R_{n+m}\right|}\leq 1,\]
and on the other hand, we have (according to Theorem \ref{T1}):
\[\log\lcm\left(R_n,R_{n+1},\dots,R_{n+m}\right)\geq \log \left|R_n R_{n+1}\cdots R_{n+m}\right|-\log \left|U_1 U_{2}\cdots U_{m}\right|-mT,\]
that is
\[\frac{\log\lcm\left(R_n,R_{n+1},\dots,R_{n+m}\right)}{\log \left|R_n R_{n+1}\cdots R_{n+m}\right|}\geq 1-\frac{\log \left|U_1 U_{2}\cdots U_{m}\right|+mT}{\log \left|R_n R_{n+1}\cdots R_{n+m}\right|}.\]
Since $m$ is fixed and $|R_n|\rightarrow +\infty$ when $n$ goes to infinity, we get
\[1\leq \lim_{n\rightarrow +\infty}\frac{\log \lcm\left(R_n,R_{n+1},\dots,R_{n+m}\right)}{\log \left|R_n R_{n+1}\cdots R_{n+m}\right|}\leq 1,\]
as required. Now, let us show that \eqref{T72} holds. According to \eqref{eee}, we have for any positive integer $t$: $\left|R_t\right|\leq \left(|a|+|b|\right)|\alpha|^t$. Next, it follows from Lemma \ref{sho} that there exist two positive constants $c_0$ and $c_1$, such that $\left|R_t\right|\geq |\alpha|^{t-c_0\log t}$ $(\forall t\geq c_1)$. Consequently, we have for any integer $t\geq c_1$:
\[|\alpha|^{t-c_0\log t}\leq \left|R_t\right|\leq \left(|a|+|b|\right)|\alpha|^t.\]
This implies that for any $n\geq c_1$, we have
\[\log\left|R_n R_{n+1}\cdots R_{n+m}\right|\leq \left(m+1\right)\log\left(|a|+|b|\right)+\frac{(m+1)(2n+m)}{2}\log |\alpha|\]
and
\[\log\left|R_n R_{n+1}\cdots R_{n+m}\right|\geq \frac{(m+1)(2n+m)}{2}\log |\alpha|+O_{m}(\log n),\]
since:
\[\log \left(n\left(n+1\right)\cdots\left(n+m\right)\right)\leq \left(m+1\right)\log \left(n+m\right)\ll_m \log n.\]
Hence:
\[\log\left|R_n R_{n+1}\cdots R_{n+m}\right|\sim_{+\infty}n(m+1)\log |\alpha|.\]
Formula \eqref{T72} then follows from the last estimate and \eqref{T71}. This completes the proof.
\end{proof}

\begin{proof}[Proof of Corollary \ref{fiii}]
This follows by applying Theorem \ref{T7} for $(P,Q)=(1,-1)$ and $(R_0,R_1)=(0,1)$.
\end{proof}

\subsection{The rest of the proofs}

\begin{proof}[Proof of Theorem \ref{co}]
By putting $R_0=0$ and $R_1=1$, we have (from Lemma \ref{eqdeR}) $R_t=U_t$ $(\forall t\in\mathbb{N})$. The left-hand side of \eqref{id} is then equal to $L_{n-k+1,n}$. Let $B_{k,n}$ denote the right-hand side of \eqref{id}. So, we have to show that $L_{n-k+1,n}=B_{k,n}$. To do so, we first show that $B_{k,n}$ divides $L_{n-k+1,n}$ and then that $L_{n-k+1,n}$ divides $B_{k,n}$. According to Theorem \ref{T1}, for any $m\in\lbrace 1,2,\dots,k\rbrace$, the number ${L}_{n-m+1,n}$ is a multiple of the integer
\[\frac{U_{n-m+1}U_{n-m+2}\cdots U_n}{{[m-1]}_{\boldsymbol{U}}!}=U_m{\binom{n}{m}}_{\boldsymbol{U}}.\]
On the other hand, the number $L_{n-k+1,n}$ is a multiple of $L_{n-m+1,n}$ ($\forall m\in\lbrace 1,2,\dots,k\rbrace$), since $n-m+1\geq n-k+1$. This implies that ${L}_{n-k+1,n}$ is a multiple of each $U_m{\binom{n}{m}}_{\boldsymbol{U}}$ ($m=1,2,\dots,k$); that is $B_{k,n}$ divides $L_{n-k+1,n}$, as required. Next, observe that for any $m\in\lbrace 1,2,\dots,k\rbrace$, we have
\[U_m{\binom{n}{m}}_{\boldsymbol{U}}=U_{n-m+1}{\binom{n}{m-1}}_{\boldsymbol{U}}.\]
Consequently, $U_{n-m+1}$ divides $B_{k,n}$ ($\forall m\in\lbrace 1,2,\dots,k\rbrace$); that is $L_{n-k+1,n}$ divides $B_{k,n}$, as required. This completes the proof. 
\end{proof}

\begin{proof}[Proof of Corollary \ref{co2}]
This is an immediate consequence of Theorem \ref{co} and the obvious identity: $U_m{\binom{n}{m}}_{\boldsymbol{U}}=U_n{\binom{n-1}{m-1}}_{\boldsymbol{U}}$ $(\forall m,n\in\mathbb{N},~1\leq m\leq n)$.
\end{proof}

\begin{proof}[Proof of Theorem \ref{tri}]
Let us prove the first part of Theorem \ref{tri}. Since $\left(U_t\right)_t$ is a divisibility sequence, we have $U_d\mid U_m$ whenever $d\mid m$ ($\forall d,m\in\mathbb{N^*}$). On the other hand, for a given positive integer $m$ and given consecutive integers $x,x+1,\dots,x+m-1$, there exists at least one $i\in\lbrace 0,1,\dots,m-1\rbrace$ such that $m$ divides $x+i$ (since: $x,x+1,\dots,x+m-1$ are pairwise distinct modulo $m$). So, by using these two facts, every $i\in\lbrace 1,2,\dots,k\rbrace$ will be a divisor of some $m\in\lbrace n-k+1,n-k+2,\dots,n\rbrace$; so, $U_i$ divides $U_m$, which is clearly a divisor of $L_{n-k+1,n}:=\lcm\left(U_n,U_{n-1},\dots,U_{n-k+1}\right)$. Hence: $\lcm\left(U_1,U_{2},\dots,U_{k}\right)$ divides $L_{n-k+1,n}$, as required. Next, the second part of Theorem \ref{tri} is equivalent to saying that:
\[{\binom{n}{k}}_{\boldsymbol{U}}\lcm\left(U_1,U_2,\dots,U_k\right)~~\text{is a multiple of}~~\lcm\left(U_n,U_{n-1},\dots,U_{n-k+1}\right);\]
that is:
\[U_{n-i}~~\text{divides}~~{\binom{n}{k}}_{\boldsymbol{U}}\lcm\left(U_1,U_2,\dots,U_k\right)~~~~(\forall i\in\lbrace 0,1,\dots,k-1\rbrace).\]
For $k=n$, the required result is trivial. Suppose for the sequel that $n>k$. So, we have to show that for any $i\in\lbrace 0,1,\dots,k-1\rbrace$, the number
\[{\binom{n}{k}}_{\boldsymbol{U}}\frac{\lcm\left(U_1,U_2,\dots,U_k\right)}{U_{n-i}}=\frac{\left(U_nU_{n-1}\cdots U_{n-i+1}\right)\left(U_{n-i-1}\cdots U_{n-k+1}\right)}{U_1,U_2,\dots,U_k}\lcm\left(U_1,U_2,\dots,U_k\right),\]
is an integer. Let $i\in\lbrace 0,1,\dots,k-1\rbrace$ be fixed. Since the numbers $\left(U_nU_{n-1}\cdots U_{n-i+1}\right)$ and $\left(U_{n-i-1}U_{n-i-2}\cdots U_{n-k+1}\right)$ are respectively multiples of ${[i]}_{\boldsymbol{U}}!$ and ${[k-i-1]}_{\boldsymbol{U}}!$ (because $\frac{U_nU_{n-1}\cdots U_{n-i+1}}{{[i]}_{\boldsymbol{U}}!}={\binom{n}{i}}_{\boldsymbol{U}}$ and $\frac{U_{n-i-1}U_{n-i-2}\cdots U_{n-k+1}}{{[k-i-1]}_{\boldsymbol{U}}!}={\binom{n-i-1}{k-i-1}}_{\boldsymbol{U}}$), it suffices to show that the number
\[\frac{{[i]}_{\boldsymbol{U}}!{[k-i-1]}_{\boldsymbol{U}}!}{{[k]}_{\boldsymbol{U}}!}\lcm\left(U_1,U_2,\dots,U_k\right)=\frac{\lcm\left(U_1,U_2,\dots,U_k\right)}{U_{k-i}{\binom{k}{i}}_{\boldsymbol{U}}}\]
is an integer. According to Theorem \ref{T1} (applied for $R_0=0$ and $R_1=1$), the number $L_{i+1,k}:=\lcm\left(U_{i+1},U_{i+2},\dots,U_{k}\right)$ is a multiple of 
\[\frac{U_{i+1}U_{i+2}\cdots U_{k}}{{[k-i-1]}_{\boldsymbol{U}}!}=\frac{{[k]}_{\boldsymbol{U}}!}{{[i]}_{\boldsymbol{U}}!{[k-i-1]}_{\boldsymbol{U}}!}=U_{k-i}{\binom{k}{i}}_{\boldsymbol{U}}.\]
The required result then follows by observing that $\lcm\left(U_1,U_2,\dots,U_k\right)$ is clearly a multiple of $L_{i+1,k}$. This completes the proof. 
\end{proof}

\begin{proof}[Proof of Corollary \ref{derrr}]
Let $n\in\mathbb{N^*}$ be fixed. From Theorem \ref{tri}, we have 
\[\lcm\left(U_1,U_{2},\dots,U_{\left\lceil\frac{n}{2}\right\rceil}\right)~~\text{divides}~~\lcm\left(U_n,U_{n-1},\dots,U_{n-\left\lceil\frac{n}{2}\right\rceil+1}\right).\]
This implies that:
\begin{align*}
\lcm\left(U_n,U_{n-1},\dots,U_{n-\left\lceil\frac{n}{2}\right\rceil+1}\right)&=\lcm\left(U_n,U_{n-1},\dots,U_{n-\left\lceil\frac{n}{2}\right\rceil+1};U_{\left\lceil\frac{n}{2}\right\rceil},U_{\left\lceil\frac{n}{2}\right\rceil-1},\dots,U_1\right)\\&=\lcm\left(U_1,U_{2},\dots,U_{n}\right),
\end{align*}
since the integers $\left(n-\left\lceil\frac{n}{2}\right\rceil+1\right)$ and $\left\lceil\frac{n}{2}\right\rceil$ are consecutive or equal. This completes the proof. 
\end{proof}

\section{Concluding remarks and open questions}

When we replace the Lucas sequence $\left(U_n\right)_n$ by the sequence $(n)_n$, the numbers ${\bi{n}{k}}_{\boldsymbol{U}}$ defined in Theorem \ref{tri} become the so-called \textit{$\lcm$-binomial numbers}. These numbers are previously studied by Farhi \cite{lcmbinom}, which raised some open problems looking at their analogy with the usual binomial coefficients. For example, Farhi \cite{lcmbinom} proved that for all positive integers $k$ and $n$ such that $n\geq k$, the number
\[\bi{n}{k}:=\frac{\lcm\left(n,n-1,\dots,n-k+1\right)}{\lcm\left(1,2,\dots,k\right)},\]
is an integer and divides $\binom{n}{k}$. But since the proof presented by Farhi investigates the $p$-adic valuation of these numbers, the later asked if we can find an alternative proof, which do not use the prime number arguments. In fact, we can easily adapt the proofs of Theorem \ref{T1} and Theorem \ref{tri} to the sequence $R_n=n$ ($\forall n\geq 1$), and getting an answer to this question. Here we list some open problems, which sometimes generalize the problems given in \cite{lcmbinom}.
\begin{enumerate}
\item We arise the general question to determine all the pairs $(k,n)$ of positive integers such that $n\geq k$ and ${\bi{n}{k}}_{\boldsymbol{U}}={\binom{n}{k}}_{\boldsymbol{U}}$.
\item From the theory of binomial coefficients, it is well known that: 
\[\binom{n}{k}=\binom{n-1}{k-1}+\binom{n-1}{k}~~~~(\forall n,k\in\mathbb{N^*},~n\geq k),\]
and then this presents a useful way to construct the binomial coefficients. We ask more generally if there exists a similar formula allowing us to construct the numbers ${\bi{n}{k}}_{\boldsymbol{U}}$.    
\item The binomial theorem tells us that: $\sum_{k=1}^{n}\binom{n}{k}x^k=(x+1)^n$. So, is there an analog formula for the sum $\sum_{k=1}^{n}{\bi{n}{k}}_{\boldsymbol{U}}x^k$?
\item In the context of Corollary \ref{cor5}, we have for $n$ sufficiently large:
\[\log \lcm\left(R_1,R_2,\dots,R_n\right)\gg\log \left|R_1R_2\cdots R_n\right|.\]
We believe that the following limit exists:
\[\lim_{n\rightarrow +\infty}\frac{\log\left|R_1R_2\cdots R_n\right|}{\log\lcm\left(R_1,R_2,\dots,R_n\right)},\]
and we ask to find its value.
\end{enumerate}

\end{document}